\documentclass[12pt]{article}
\usepackage{amsmath,amsfonts,amsthm,amssymb,amscd,colordvi}
\usepackage{authblk}

\usepackage{color}

\newcommand{\diver}{\mathop{\rm div}\nolimits}

\newcommand{\p}{\partial}

\newcommand{\e}{\varepsilon}
\newcommand{\vk}{\varkappa}
\newcommand{\vkj}{\varkappa^{(j)}}

\newcommand{\cI}{\mathcal I}

\newcommand{\R}{{\mathbb R}}

\newcommand{\Z}{{\mathbb Z}}

\newcommand{\cE}{{\cal E}}
\newcommand{\T}{{\mathbb T}}
\newcommand{\N}{{\mathbb N}}

\newcommand{\es}{{\mathbb S}}

\newcommand{\Om}{{\cal O}}

\newcommand{\cF}{{\cal F}}

\newcommand{\cH}{{\cal H}}

\newcommand{\cS}{{\cal S}}

\newcommand{\dist}{\mathop{\rm dist}\nolimits}
\newcommand{\vf}{vectorfield   }
\newcommand{\rot}{\mathop{\rm rot}\nolimits}

\newcommand{\Vect}{\text{\upshape Vect}}

\newcommand{\SVect}{\text{\upshape SVect}}

\def\12{\tfrac12}

\theoremstyle{plain}
\newtheorem{theorem}{Theorem}[section]
\newtheorem{lemma}[theorem]{Lemma}
\newtheorem{proposition}[theorem]{Proposition}
\newtheorem{corollary}[theorem]{Corollary}
\theoremstyle{definition}
\newtheorem{definition}[theorem]{Definition}
\theoremstyle{remark}
\newtheorem{remark}[theorem]{Remark}
\theoremstyle{example}
\newtheorem{example}[theorem]{Example}
\newtheorem{convention}[theorem]{Convention}
\theoremstyle{definition}

\numberwithin{equation}{section}

\newcounter{bk}

\begin{document}

\title{KAM theory and the 3D Euler equation}

\author[1]{Boris Khesin\thanks{khesin@math.toronto.edu}}
\author[2]{Sergei Kuksin\thanks{kuksin@gmail.com}}
\author[3]{Daniel Peralta-Salas\thanks{dperalta@icmat.es}}
\affil[1]{Department of Mathematics, University of Toronto, Toronto, ON M5S 2E4, Canada}
\affil[2]{Universit\'e Paris-Diderot (Paris 7), UFR de Math\'ematiques - Batiment Sophie Germain, 5 rue Thomas Mann, 75205 Paris CEDEX 13, France}
\affil[3]{Instituto de Ciencias Matem\'aticas, Consejo Superior de Investigaciones Cient\'ificas, 28049 Madrid, Spain}

\maketitle

\date{}
\maketitle

\begin{abstract}
We prove that the dynamical system defined by the hydrodynamical Euler equation on any closed Riemannian $3$-manifold $M$ is not mixing in the $C^k$ topology ($k>4$ and non-integer) for any prescribed 
 value of helicity and sufficiently large values of energy. This can be regarded as a 3D version of Nadirashvili's and Shnirelman's theorems showing the existence of wandering solutions for the 2D Euler equation. Moreover, we obtain an obstruction for the mixing under the Euler flow of $C^k$-neighborhoods of divergence-free vectorfields on $M$. On the way we construct a family of functionals on the space of divergence-free 
$C^1$ vectorfields on the manifold, which are integrals of motion of the 3D Euler equation. 
Given a vectorfield these functionals measure the part of the manifold foliated 
by ergodic invariant tori of  fixed isotopy types. We use the KAM theory to establish some 
 continuity properties of these functionals in the $C^k$-topology. This allows one to get a lower bound for the $C^k$-distance between a divergence-free vectorfield (in particular, a steady solution) and a trajectory of the Euler flow. 
\end {abstract}
\bigskip

\section{Introduction}
\label{s0}
One of the achievements of the KAM theory was establishing that a typical Hamiltonian 
system close to a completely  integrable one has many invariant tori and hence cannot be ergodic. 
On the other hand, the celebrated Arnold's theorem on the structure of typical 3D steady flows 
of an ideal fluid proves that such flows are almost everywhere fibered by invariant tori. In this paper we 
show how this similarity of steady flows and integrable systems implies non-ergodicity of 
the infinite-dimensional dynamical system defined by the hydrodynamical Euler equation.

Namely, the motion of an ideal fluid on a Riemannian closed manifold $M$ is described by its velocity field $u(\cdot,t)$, which satisfies the Euler equation
\begin{equation} \label{4.1}
\partial_tu+ \nabla_u u=-\nabla p\,, \;\; \diver u=0\,,
\end{equation}
for a pressure function $p(\cdot,t)$ defined by these equations up to a constant. Here $\nabla_u u$ is the covariant derivative of $u$ along itself. This equation implies that the \emph{vorticity field} $\omega:=\rot u$ is transported by the flow, i.e.
\begin{equation}\label{eqtrans}
\partial_t\omega+[u,\omega]=0\,,
\end{equation}
and hence in any flow the vortex lines at $t=0$ are diffeomorphic to the corresponding 
vortex lines at any other $t$ (for which the solution exists). 
This phenomenon is known as Kelvin's circulation theorem.

A solution $u$ to the Euler equation is called \emph{steady (or stationary)} when it does not depend on time, so it satisfies the equation 
$$
\nabla_u u=-\nabla p\,, \;\; \diver u=0\,.
$$
In particular, Eq.~\eqref{eqtrans} implies that the vorticity and the velocity fields commute for stationary solutions, that is $[\omega,u]=0$. The topology of ``typical" steady solutions of 3D Euler flows was described by Arnold in his structure theorem~\cite{Arn66}. Namely, under the assumption of sufficient smoothness and non-collinearity of $u$ and $\rot u$, the manifold $M$, away from a singular set, is fibred by 2-tori invariant for both fields $u$ and $\rot u$. The motion on these tori is periodic or quasi-periodic. Accordingly, a steady flow in 3D
looks like an integrable Hamiltonian system with two degrees of freedom on a fixed energy level. 

Inspired by the mechanism of vorticity transport and the existence of many invariant tori of the vorticity in steady fluid flows, we introduce an integral of motion
for the Euler equation that is independent of the energy and the helicity (the classical first integrals of 3D Euler). This conserved quantity is a functional $\vk$ on the space of divergence-free vectorfields which measures the fraction of $M$ which is covered by ergodic invariant tori of $\rot u$. Moreover,
the way how the invariant tori are embedded (knotted) in $M$ is also an invariant
and it gives a family  $\vk_a$, $a\in\Z$, of infinitely many conserved quantities. 
(More precisely, the index $a$ belongs to the countable set of different embedding classes.) 
The whole family of quantities $\{\vk_a\}$ for a given divergence-free vectorfield will be called the integrability spectrum of this field on the manifold. This setting provides a framework to apply the KAM theory, which in fact allows us to prove basic continuity properties of $\vk_a$ evaluated at certain nondegenerate vectorfields.  

Roughly speaking, the key idea is as follows. The fraction of $M$ filled in with invariant tori of the vorticity does not change during the evolution as a consequence of the vorticity transport. If the vorticity of a non-stationary solution approaches an integrable (i.e. a.e. fibered by tori) divergence-free field $v$, then the KAM theory of divergence-free vectorfields guarantees that the time-dependent vorticity must possess many invariant tori of the same topology as those for $v$, under certain nondegeneracy conditions. This allows us to estimate how close a non-stationary solution of the Euler equation can get to a given integrable vectorfield, and in particular to a steady solution.

Using the conserved quantities $\vk_a$, the KAM theory, and simple differential topology, we prove the main results of this article: First, \emph{the Euler flow~\eqref{4.1} is not ergodic in the $C^k$-topology, $k>4$ and non-integer\footnote{We assume that $k$ is a non-integer so that the Euler equation defines a local flow in the H$\ddot{\text{o}}$lder space $C^k$, cf.~\cite{EM70}. This property fails for integer $k$, see~\cite{Bo}.
}, on the space of divergence-free vectorfields of fixed  helicity and sufficiently large energy}, see Theorem~\ref{thm:main}. This property is a  3D version of Nadirashvili's and Shnirelman's theorems on the existence of wandering solutions for the  Euler equation on a 2D annulus~\cite{Na91} and on the $2$-torus~\cite{Sh97}. Second, \emph{if two divergence-free vectorfields
have different integrability spectra and their vorticities are integrable,
then any sufficiently small $C^k$-neighborhoods of these vectorfields do not mix under the Euler flow}, see Theorem~\ref{Thlast}. 
Both results hold for any closed Riemannian manifold, so the metric does not play a relevant role, in particular negative curvature does not imply mixing for the 3D Euler flow in the $C^k$-topology. 

In the particular case of $\T^3$ and $\es^3$, we prove the existence of open domains of divergence-free vectorfields with fixed energy and helicity, in the $C^k$-topology ($k>4$), which cannot approach certain steady solutions under the evolution of the Euler equation,  and in the case of $\T^3$ we show that there are pairs of steady states that cannot be joint by a heteroclinic connection.

We would like to remark that, while any smooth invariant
of the vorticity is a conserved quantity (a Casimir  functional) of the Euler flow on
account of Kelvin's circulation theorem, we are not aware of any use
of this invariant to study the asymptotic behavior of the 3D Euler
equation. In particular, the integrability spectrum $\{\vk_a\}$ that we introduce in this paper is
especially suited to give information on the Euler flow near steady
states. This approach might be useful in constructing other conserved
quantities  in order to analyze the 3D Euler equation combining tools
from dynamical systems and PDEs.

The paper is organized as follows. In Section~\ref{s1} we recall some basic facts of exact divergence-free vectorfields and Hodge theory of closed manifolds. A KAM theorem for divergence-free vectorfields satisfying appropriate nondegeneracy conditions is stated in Section~\ref{s2}. The  conserved quantities $\vk$ and $\vk_a$, their main properties and some intermediate constructions are introduced in Sections~\ref{s3} and \ref{isotopy}, where we also give examples of steady solutions to the Euler equation which are integrable and nondegenerate. Finally, in Section~\ref{s4} we apply the previously developed machinery to prove some non-mixing properties of the Euler flow and to estimate the distance between time-dependent solutions of the Euler equation and nondegenerate divergence-free vectorfields.

\bigskip
 
{\bf Acknowledgments.} This paper was conceived when one of us (B.K.)
was visiting \'Ecole Polytechnique in 2011. We are very grateful for comments and discussions to Y.~Eliashberg, B.~Fayad, J.~Fern\'andez de Bobadilla, L.~Polterovich, F.~Presas,  M.~Sevryuk,  and A.~Shnirelman. We also thank two reviewers for valuable suggestions and corrections. The research of B.K. was partially supported by CNRS and NSERC research grants. B.K. is also 
grateful to the Max-Planck Institute in Leipzig for support and kind hospitality.
S.K. was  supported by l'Agence Nacionale de la Recherche through the grant ANR-10-BLAN 0102. D.P.-S was supported by the ERC Starting Grant 335079 and the Spanish MINECO grants MTM2010-21186-C02-01 and SEV-2011-0087.

\bigskip


\section{Divergence-free vectorfields on 3D Riemannian manifolds}   \label{s1}
All along this paper $M$ is a smooth ($C^\infty$) closed 3D manifold endowed with 
a smooth  Riemannian metric $(\cdot,\cdot)$ and the corresponding volume form $\mu$.
We shall assume this form to be normalized in such a way that the total volume of $M$ equals 1, that is
$$\int_M\mu=1\,.$$
The measure of a subset $U$ of $M$ with respect to the volume form $\mu$ will be denoted by $\text{meas}\,(U)$.
\begin{convention}
Since we shall consider analytic ($C^\omega$) functions in some parts of the paper, we establish the convention that if a function on a manifold $M$ is said to be analytic, then the manifold itself and the volume form $\mu$ are assumed to be analytic as well.
\end{convention}
A \vf $V$ on $M$ is called {\it divergence-free} or {\it solenoidal} (with respect to the volume form 
$\mu$), and we write $\diver V=0$, if the Lie derivative of the volume form along $V$ vanishes, 
i.e. $L_V\mu=0$, or equivalently, since $L_V=i_Vd+di_V$, if the 2-form $i_V\mu$ is closed. 
The field $V$ is called {\it exact divergence-free} or {\it globally solenoidal} (with respect to $\mu$)  
if the  2-form $i_V\mu$ is exact, cf.~\cite{AKh}. In local coordinates $(x,y,z)$, the volume form reads as $\mu=p(x,y,z)dx\wedge dy\wedge dz$ for some positive function $p$, and the divergence-free condition is written as
\begin{equation}\label{eqdivc}
\frac{\partial(pV_x)}{\partial x}+\frac{\partial(pV_y)}{\partial y}+\frac{\partial(pV_z)}{\partial z}=0\,.
\end{equation}

For a \vf  $V$  we denote by $V^\flat$ its dual 1-form, corresponding to
 $V$ with respect to the Riemannian structure, i.e., 
$(V, W)=V^\flat(W)$ for any \vf $W$ on $M$. 
It is well known~\cite{W} that a \vf $V$ is divergence-free if
and only if the 1-form $V^\flat$ is coclosed, i.e., $d^* V^\flat=0$, where $d^*$ is the codifferential operator. 
Recall that the {\it gradient} of a function $f$ on $M$ is a \vf $\nabla f$ defined by
$(\nabla f)^\flat=df$.  
 The  {\it vorticity}  field $U:=\rot V$ of a \vf $V$   is  defined by the relation
\begin{equation} \label{1.1}
i_U\mu=d(V^\flat).
\end{equation}
Clearly $U=\rot V$   is an exact divergence-free vectorfield, and $\rot \circ \nabla \equiv 0$.

\begin{example} \label{ex1.1}
{\rm
Consider the 3-torus $M=\T^3 = (\R/2\pi\Z)^3$ endowed with the flat metric, so that $\mu=dx\wedge dy\wedge dz$. Then a \vf 
 $V= f \partial_x+ g \partial_y+h  \partial_z$ is divergence-free if 
$\ \p f/\p x+\p g/\p y+\p h/\p z=0$, and is exact 
 divergence-free if, in addition, 
$$\int_M f\,dx\wedge dy\wedge dz=\int_M g\,dx \wedge dy \wedge dz=\int_M h\,dx\wedge dy\wedge dz= 0\,.$$ 
Indeed, the divergence-free condition for $V$ is equivalent to closedness of the 2-form 
 $i_V\mu$, while to be exact this 2-form  has to give zero when integrated against 
 any closed 1-form over $\T^3$. The above three relations are equivalent to
 $\int i_V\mu\wedge dx=\int i_V\mu\wedge dy=\int i_V\mu\wedge dz=0$.
In fact, the above condition of zero averages means that 
the divergence-free field $V$ ``does not move the mass center", and hence is 
exact on the torus. Actually, the averages of the functions $f,g,$ and $h$ represent 
the cohomology class of the field $V$, or equivalently, of the corresponding 2-form $i_V \mu$.  
 Also note that, in these flat coordinates, the vorticity of $V$  is given by the classical 
relation $\rot V=\nabla\times V$.
}
\end{example}

\smallskip

Let $C^k(M), k\ge 0$, be the H\"older space of order $k$ of functions on $M$ 
(for $k\in\N$ this  is the  space of $k$ times continuously differentiable functions), 
and $\Vect^k(M)$ be the space of vectorfields on $M$ of the same smoothness. 
  For $k\ge1$,  by $\SVect^k$ and $\SVect^k_{ex}$ we denote 
 the closed subspaces  of $\Vect^k (M)$, formed by the divergence-free  
 and exact divergence-free vectorfields, respectively.   
 The Helmholtz decomposition for vectorfields  is dual to the Hodge decomposition 
 for 1-forms (see~\cite{W, Taylor}) under 
 the duality  $V\mapsto V^\flat$. It states that any \vf $V\in \Vect^k, k\ge1$, can be uniquely decomposed into the sum 
 \begin{equation} \label{1.2}
V=\nabla f +W+\pi,
\end{equation}
where $W\in \SVect^k_{ex}$ and $\pi\in {\cal H}\subset\Vect^k$ is a harmonic vectorfield.  
The latter means that $\Delta\pi^\flat=0$ where $\Delta$ is the Hodge Laplacian or, equivalently, $d\pi^\flat=0$ and $d^*\pi^\flat=0$ (see~\cite{W}). It is easy to check that each vectorfield in this decomposition is $L^2$-orthogonal to the other components. By the Hodge theory the harmonic 
vectorfields are smooth; they form a finite-dimensional subspace of $\Vect^\infty(M)$, independent of $k$, 
whose dimension equals  the first Betti number of $M$. 
Moreover, the projections of $V$ onto $\nabla f$, $W$ and $\pi$ 
are continuous operators in $\Vect^k(M)$ if $k$ is not an integer. 
 A \vf $V$ is divergence-free if and only if its gradient component vanishes, i.e.  $\nabla f\equiv0$ in Eq.~\eqref{1.2}. 

In the following lemma we state some properties of the operator $\rot$ which will be useful later. The result is well known, but we provide a proof for the sake of completeness. 
 
 \begin{lemma}\label{l11}
 For $k\ge2$ the vorticity operator defines a continuous  map
 \begin{equation} \label{1.3}
 \rot: \SVect^k(M) \to  \SVect^{k-1}_{ex}(M)\,.
 \end{equation}
 If $k>2$ is not an integer, then the map is surjective and its
  kernel  is formed by harmonic vectorfields. 
 \end{lemma}

 \begin{proof}
 In local coordinates $\rot$ is a first-order differential operator. 
 Since its image is formed by exact divergence-free 
 vectorfields, then for $k\ge2$ this map defines a continuous linear operator~\eqref{1.3}.  Let $U \in\SVect^{k-1}_{ex}$ be an exact divergence-free vectorfield.  Then $U$  satisfies Eq.~\eqref{1.1} with a suitable 1-form $\bar V^\flat$.  
 If $k$ is not an integer, then by the Hodge theory, $\bar V^\flat$ is $C^k$-smooth~\cite{W}, whence  
 $\bar V\in \Vect^k$.  
 Consider the decomposition~\eqref{1.2} for the field  $\bar V=\nabla f +W+\pi$. 
Since $\rot\nabla f=0$, take a new field $V=W+\pi\in \SVect^k$. By construction, 
$\rot V=\rot \bar V=U$, i.e.  the mapping \eqref{1.3} is surjective.
 
 Let $V$ belong to the kernel of~\eqref{1.3}. Since $\diver V=0$, then 
  in the decomposition~\eqref{1.2} we have $\nabla f=0$. As we explained above, 
 $\rot\pi=0$, so $\rot W=0$.  Then $d W^\flat=0$ and $d^* W^\flat=0$, hence being $L^2$-orthogonal to harmonic forms, this implies that $W=0$. Accordingly, $V=\pi$ is a harmonic vectorfield. Since any such \vf is divergence-free, the lemma is proved. 
\end{proof}

\bigskip


\section{A KAM theorem for divergence-free vectorfields.
}   \label{s2}

Let $V$ be a $C^{k}$ divergence-free  \vf ($k\ge 1$) on a closed $3$-manifold $M$ endowed with a volume form $\mu$, and let $T^2\subset M$ be an invariant 2-torus of $V$ of class $C^{k}$. We assume that in a  neighbourhood $\Om$ of the torus $T^2$, one can construct $C^{k}$-coordinates $(x,y,z)$, 
where $(x,y)\in \T^2=(\R/ 2\pi\Z)^2$ and  $z\in(-\gamma, \gamma)$ $(\gamma>0$), such that $\mu|_{\Om}=dx\wedge dy\wedge dz$ and 
$T^2=\{z=0\}$.

\begin{remark}
An invariant torus $T^2$ can be embedded in a $3$-manifold $M$ in many non-equivalent ways. In this and the next section the embedding of $T^2$ is not relevant since our analysis is in a tubular neighborhood $\Om$ of the torus, which is diffeomorphic to $\T^2\times(-\gamma,\gamma)$ independently of the embedding. The different ways an invariant torus can be embedded in $M$ will be exploited later. \end{remark}

Next we state a KAM theorem for divergence-free vectorfields, for which we assume that there are functions $f(z), g(z)$ of class $C^{k}$ defined on $(-\gamma, \gamma)$, and a Borelian subset $Q\subset(-\gamma,\gamma)$, satisfying the following  KAM nondegeneracy conditions:
\begin{enumerate}
\item For each $z\in Q$ we have $f^2(z)+g^2(z)\ne0$, 
and the torus $\T^2\times \{z\}$ is invariant for the vectorfield $V$, which assumes the form (for this value of $z$): 
\begin{equation}\label{1}
\dot x =f(z),\quad \dot y=g(z),\quad \dot z=0.
\end{equation}
\item The Wronskian of $f$ and $g$ is uniformly bounded from 0 on $Q$, i.e. there is a positive $\tau$ such that for each $z\in Q$ we have the twist condition:
\begin{equation}\label{2}
|f'(z)g(z)-f(z)g'(z)|\geq \tau>0 \,.
\end{equation} 
\end{enumerate}

Observe that Condition $1$ above implies that $i_V(dx\wedge dy\wedge dz)$ is exact in the domain $\Om$ (in particular, $V$ is divergence-free), provided that $Q=(-\gamma,\gamma)$. Indeed, 
$$
\alpha:=i_V(dx\wedge dy\wedge dz)=-g(z)dx\wedge dz+f(z)dy\wedge dz\,,
$$
so $\alpha=d\beta$, where $\beta$ is the $1$-form
$$
\beta=\Big(\int_0^zg(s)ds\Big)dx-\Big(\int_0^zf(s)ds\Big)dy\,.
$$

We also note that if $V$ satisfies Conditions~1 and~2 and is divergence-free with respect to a  volume
form $p(x,y,z)dx\wedge dy\wedge dz$, then it is easy to check that we must have $p=p(z)$ (see Eq.~\eqref{eqdivc}). Modifying the coordinate $z$ to a suitable
$\tilde z(z)$ we achieve that $\mu|_{\Om}=dx\wedge dy\wedge d\tilde z$, and  Conditions~1 and~2 still hold. So
the assumption that $\mu|_{\Om}=dx\wedge dy\wedge d z$ in  coordinates $(x,y,z)$ is in fact a consequence of Conditions~1 and~2.

Consider an exact divergence-free vector field $W$ of class $C^{k}$ and denote by $\e:=\|V-W\|_{C^{k}}$ the $C^{k}$-distance between $V$ and $W$.

\begin{theorem}\label{t1}
Assume that the divergence-free vectorfield $V$ satisfies the previous assumptions 1-2 with $Q=(-\gamma,\gamma)$. Then there are real numbers $k_0$ and $\e_0=\e_0(V)>0$ such that if $k>k_0$ and $\e<\e_0$, there exists a $C^1$-diffeomorphism  $\Psi:(x,y,z)\mapsto ({\bar x},{\bar y},{\bar z})$, preserving the volume $\mu$, a  Borelian set $\bar Q\subset (-\gamma,\gamma)$, and 
$C^{1}$-functions ${\bar f}(z),{\bar g}(z)$ such that 
\begin{itemize}
\item meas$\, ((-\gamma,\gamma) \setminus \bar Q\big)\leq C\tau^{-1}\sqrt{\epsilon}$ as $\e\to0$, where $C$ depends on $k$ and the $C^k$-norm of $V$, and $\tau$ is defined in the inequality~\eqref{2}.
\item $ \|\Psi-\,$id$\|_{C^1}\to0$ as $\e\to0$,
\item $ \|f-{\bar f}\|_{C^{1}} + \|g-{\bar g}\|_{C^{1}} \to0$ as $\e\to0$,
\item for $\bar z\in \bar Q$ the \vf $W$ transformed by the diffeomorphism $\Psi$ assumes the form
\begin{equation*} 
\dot {\bar x} ={\bar f} ({\bar z}),\quad \dot {\bar y}={\bar g}({\bar z}),\quad \dot {\bar z}=0,
\end{equation*}
and the ratio  $(\bar f/\bar g)(\bar z)$ is an irrational number. 
\end{itemize}
\end{theorem}

Herman's theorem~\cite{He83} on the class of differentiability, for which Moser's twist
theorem holds, implies that one can take
 $k_0=3$. By~\cite{He83}  and 
 a recent theorem of Cheng and Wang~\cite{C-W}, for closely related KAM-results the $C^3$-smoothness is sharp, so it is very plausible that 
 for $k_0<3$ the assertion of Theorem~\ref{t1} is false. In any case, Theorem~\ref{t1} holds for any
\begin{equation}\label{Mos}
k> 3\,. 
\end{equation}
For the case of Hamiltonian systems, the reader can consult~\cite{Pos,Sal}. (Note that the smoothness of $V$ and $W$ corresponds in~\cite{Pos, Sal} not to the smoothness of the Hamiltonian $H$, but to that of the Hamiltonian vectorfields $J\nabla H$.) For the case when the vectorfield $V$ is analytic, this result is proved in~\cite{Do82}.

\begin{remark}
The KAM theorem stated above can be proved by 
taking a  Poincar\'e section and reducing the problem to area-preserving perturbations of a twist map in the annulus in that section. Indeed, assume that $g(z)\ne0$ for any $z\in(-\gamma,\gamma)$ in (\ref{1}) (the case $f(z)\ne0$ is similar) and
consider the Poincar\'e section $\{y=const\}$. On account of Condition~1, the Poincar\'e map $\Pi_0$ of $V$ is a diffeomorphism of the annulus $\R/ 2\pi\Z\times (-\gamma,\gamma)$ given by
$$
(x,z)\mapsto \Big(x+\frac{2\pi f(z)}{g(z)},z\Big)\,,
$$
which is a nondegenerate twist map as a consequence of Condition~2. Obviously, the map $\Pi_0$ preserves the area form $dx\wedge dz$, which is, generally speaking, different from the area form $A_W$ preserved by the perturbed Poincar\'e map $\Pi_\epsilon$ associated to the vector field $W$. This problem can be overcome by using Moser's trick~\cite{Mo65}, which ensures that there exists a diffeomorphism $\Phi_\epsilon$ of the annulus, so that $\Phi_\epsilon^*A_W=dx\wedge dz$, 
and this diffeomorphism is close to the identity because $\Pi_0$ and $\Pi_\epsilon$ are close.
Then the theorem  follows by applying Moser's twist theorem~\cite{He83} to the map $\Phi_\epsilon^{-1}\circ \Pi_\epsilon\circ \Phi_\epsilon$, which is conjugated to $\Pi_\epsilon$. Observe that the perturbed Poincar\'e map $\Pi_\epsilon$ is exact, so that Moser's theorem is applicable, see Remark~\ref{R:exactP} below. The details are left to the interested reader. Notice that the use of a Poincar\'e section in the proof implies that the result does not depend on the 
 parametrization of the vectorfield.  
\end{remark}

\begin{remark}\label{R:exactP}
The condition on the exactness of the divergence-free vectorfield $W$ in Theorem~\ref{t1} is automatically satisfied if $W$ is a vorticity field. This assumption is crucial for the KAM theorem to hold. Indeed, if the $2$-form $i_W\mu$ is exact, the Stokes theorem implies that the flux of the vectorfield $W$ across any closed surface is zero, and in particular 
$$
\int_{T^2}i_W\mu=0\,,
$$
which implies the exactness of the Poincar\'e map, a necessary condition for applying Moser's twist theorem. 
\end{remark}

Other KAM theorems in the context of volume-preserving maps have been obtained in~\cite{C-S}, while for divergence-free vectorfields the reader can consult~\cite{Sevr,BHT}. These references work in the analytic ($C^\omega$) setting. The fact that their assertions remain 
true for vectorfields of finite smoothness follows from the reduction to Moser's twist theorem,   explained in Remark~3.3.

\bigskip


\section{A measure of non-integrability of divergence-free vectorfields.}   \label{s3}

The main goal of this section is to define a functional on the set of exact divergence-free vectorfields which measures how far a \vf $V$ is from an integrable nondegenerate vectorfield. 
We start by introducing a precise definition of integrability,  inspired by the celebrated Arnold structure theorem for 3D steady flows~\cite{Arn66} (see also~\cite[Section II]{AKh}).  

\begin{definition}\label{def}
(1) If $V$ is a divergence-free vectorfield with an invariant domain $\Om_j\cong\T^2\times (-\gamma_j,\gamma_j)$ covered by invariant tori of $V$, and $V$ satisfies condition $1$ in Section~\ref{s2} for $z\in Q_j=(-\gamma_j,\gamma_j)$, we say that $V$ is \emph{canonically integrable on} $\Om_j$.
\\  
(2) A divergence-free
 $C^1$-\vf $V$ on $M$  is called {\it  (Arnold) integrable}, if there is a closed subset 
$G\subset M$, $\rm{meas}\,(G)=0$, such that its complement $M\setminus G$ 
is a union of a finite or countable system 
of $V$-invariant domains $\Om_j$, where $V$ is canonically integrable. If the system of domains $\Om_j$ is finite, $V$ is called  {\it well integrable}.
\\
(3) An  Arnold integrable \vf $V$ is called  {\it nondegenerate} if the system of domains $\Om_j$ can be 
chosen in such a way that condition $2$ in Section~\ref{s2} holds for each $j$.\\
(4) 
 If the set $G$ introduced above is any Borel subset of $M$ and we do not 
require $Q_j$ to be the whole interval  $(-\gamma_j,\gamma_j)$, 
 then $V$ is called {\it partially integrable on $M$} and 
{\it integrable on $M$ outside $G$ with ``holes", corresponding to 
$\T^2\times \big((-\gamma_j,\gamma_j)\setminus Q_j\big)$, $j\ge1$. }
\end{definition} 

We recall that an invariant torus $T^2$ of a vectorfield $V$ on $M$ is called \emph{ergodic} if the field is nonvanishing on the torus and some trajectory of $V$ is dense in $T^2$. Clearly, an invariant torus $T^2\subset \Om_j$, $T^2\cong \T^2\times\{z\}$, $z\in Q_j$, as in Definition~\ref{def} above, is ergodic if the number $f(z)/g(z)$ is finite and irrational. 

\begin{example} \label{ex3.2}
{\rm
Let $u$ be a steady solution of the 3D Euler equation~\eqref{4.1} in $M$, i.e.
$$
\nabla_u u=-\nabla p,
\;\;\; \diver u=0
\qquad\text{in}\quad M\,.
$$
This equation can be rewritten as 
$$
i_\omega i_u\mu=d\alpha\,,
$$
where $\omega:=\rot u$ is the {\it vorticity}, $\alpha:=p+|u|^2/2$ is the {\it Bernoulli function} and $\mu$ is the volume form on $M$. It is immediate that the Bernoulli function is a first integral for both the vectorfields $u$ and $\omega$. The regular level sets of $\alpha$ must be 2-tori, since they admit non-vanishing vectorfields tangent to them, see \cite{Arn66} or \cite[Section II]{AKh}. Assume that 
the set of critical points for the function $\alpha$ has zero measure.\footnote{Note that this 
non-degeneracy property fails for the important class of {\it Beltrami solutions}, defined by the 
equation $\rot  u=\lambda u$ ($\lambda$ is a constant).
}  Then the vorticity $\omega$ is Arnold integrable in $M$.
}
\end{example}

The minimal set $D$ for which $V$ is integrable on  $M\setminus D$  
measures the non-integrability of $V$. To develop this idea we define the following functional:

\begin{definition}
The \emph{partial integrability functional} $\vk$ 
on the space of $C^1$  exact divergence-free vector fields
$$
\vk:\SVect^1_{ex}(M)\to [0,1]\,
$$
assigns to a vectorfield $V\in \SVect^1_{ex}(M)$ the inner measure\footnote{We recall that the inner measure of a set is the supremum of the measures of its compact subsets. We use here the inner measure rather than the Lebesgue 
measure to avoid the delicate issue of whether the union of ergodic invariant tori is measurable or not (cf. also Definiton~\ref{defka}).} of the set
 equal to the union of all  ergodic 
$V$-invariant two-dimensional $C^1$-tori.
 Since the total measure of $M$ is normalized by $1$,  then $\vk\in[0,1]$. 
\end{definition}

The functional $\vk$ does not distinguish between different isotopy classes of invariant tori. In the following section we shall define a version of this functional taking into account the way in which invariant tori are embedded in $M$. Before describing properties of $\vk$, let us provide explicit examples of stationary solutions to 3D Euler whose vorticities 
are integrable and nondegenerate. The examples are on $\T^3$ and $\es^3$, with the canonical metrics. 

\begin{example}(Existence of steady solutions of the Euler equation in $\T^3$ whose vorticities are integrable and nondegenerate). \label{extorus}
\rm{
Consider the divergence-free vectorfield $u^z$ defined by
$u^z= f(z)\partial_x+ g(z)\partial_y$, where $f$ and $g$ are analytic nonconstant $2\pi$-periodic functions. The function $z$ is a first integral of $u^z$, hence the trajectories of this vectorfield are tangent to the tori $T_c:=\{z=c\}$, and on
each torus the field is linear.
The same happens with the vorticity $\rot u^z= -g'(z) \partial_x + f'(z) \partial_y$. Now note that:
\begin{enumerate}
\item The fields $u^z$ and $\rot u^z$ commute, $[u^z, \rot u^z]=0$, and the Bernoulli function is given by $\alpha=\frac{1}{2}(f^2+g^2)$. This implies that $u^z$ is
a solution of the steady Euler equation on $\T^3$, cf. Example~\ref{ex3.2}.
\item For generic $f$ and $g$ the field $\rot u^z$ satisfies the nondegeneracy conditions $1$ and $2$ of Section~\ref{s2} everywhere except for finitely many values of $z$. 
\end{enumerate}
Therefore, one has $\vk(\rot u^z)=1$.  Note that all the invariant tori of $\rot u^z$ are in one and the same isotopy class, which is nontrivial, because the tori $T_c$ are  homologically nontrivial.
Similarly, one can construct steady solutions  $u^x$ and $u^y$
of the Euler equation on $\T^3$ whose invariant tori are given by $\{x=c\}$ or $\{y=c\}$, and hence not isotopic to the tori $\{z=c\}$.
}
\end{example}

\begin{example}(Existence of steady solutions of the Euler equation in $\mathbb S^3$ whose vorticities are integrable and nondegenerate.)\label{exsphere}
\rm{
It is convenient to represent $\mathbb S^3$ as the set of points $\{(x,y,z,w)\in\R^4:x^2+y^2+z^2+w^2=1\}$. Consider the Hopf fields $u_1$ and $u_2$ on $\es^3$ that satisfy the equations $\rot u_1=2u_1$ and $\rot u_2=-2u_2$, and hence they are divergence-free. In coordinates these fields read as $u_1=(-y,x,w,-z)|_{\es^3}$ and $u_2=(-y,x,-w,z)|_{\es^3}$. It is evident that the function $F:=(x^2+y^2)|_{\es^3}$ is a first integral of both $u_1$ and $u_2$, and its regular level sets are tori, so these vectorfields are well integrable. Since all the trajectories of $u_i$ are periodic, they are completely degenerate. It is not difficult to check the following properties:
\begin{itemize}
\item $(u_1,u_1)=(u_2, u_2)=1$ and $(u_1, u_2)=2F-1$, where $(\cdot,\cdot)$ is the scalar product on $\es^3$ with respect to the round metric.
\item $u_1\times u_2=-\nabla F$, where $\times$ and $\nabla$ are the vector product and the gradient operator, respectively, on $\es^3$ with respect to the round metric. 
\end{itemize}    
Now define the vectorfield
$$
u:=f(F)u_1+g(F)u_2\,,
$$
where $f,g$ are analytic functions. Of course, the field $u$ is divergence-free because $F$ is a first integral of $u_i$, and  it is non-vanishing whenever $f^2+g^2\neq 0$ because $\{u_1,u_2\}$ defines a basis on each level set of $F$. After a few straightforward computations we get
\begin{align*}
&\rot u=[f'(2F-1)+2f+g']u_1-[g'(2F-1)+2g+f']u_2\,,\\
&u\times \rot u=[ff'+gg'+4fg+(2F-1)(fg'+gf')]\nabla F\,.
\end{align*}
Therefore, defining $H(F):=ff'+gg'+4fg+(2F-1)(fg'+gf')$, we conclude that $u$ is a steady solution of the Euler equation on $\es^3$ with Bernoulli function $\alpha=\int_0^FH(s)ds$, cf. Example~\ref{ex3.2}. The vorticity $\rot u$ is well integrable ($F$ is a first integral) and for generic choices of $f$ and $g$ it is nondegenerate. Therefore, $\vk(\rot u)=1$. Moreover, all the invariant tori of $\rot u$ are in the same isotopy class, which is in fact the trivial one because the tori are unknotted.   
}
\end{example} 

The proposition below summarizes the main elementary properties of the partial integrability functional $\vk$.

\begin{proposition}\label{p}
The partial integrability functional $\vk$ satisfies the following properties:
\begin{enumerate}
\item If $V$ is partially integrable and for some $j$ and $z\in Q_j$ which is a
 point of density for $Q_j$ the corresponding functions $f$ and $g$ satisfy the twist condition~\eqref{2}, then $\vk(V)>0$.
\item If $V$ is Arnold integrable and nondegenerate, then $\vk(V)=1$.
\item If $\Phi$ is a volume-preserving $C^2$-diffeomorphism of $M$, then $\vk(V)=\vk(\Phi^*V)$.
\item If all the trajectories of $V$ in the complement of an invariant zero-measure subset of $M$ are periodic, then $\vk(V)=0$. The same result holds if $V$ has two first integrals which are independent almost everywhere on $M$. 
\item If a domain $\Om\subset M$ is $V$-invariant and all the
 trajectories of $V$ in $\Om$ have positive maximal Lyapunov exponent, then $\vk(V)\le1-\,$meas$\,(\Om)$.
\item Let $\xi(t)$ be a trajectory of $V$ and denote its closure in $M$ by $\cF$. Then $\vk(V)\le 1- \,$meas$\, (\cF)$.
\end{enumerate}
\end{proposition}

\begin{proof}
Statements 1-5 follow almost straightforwardly from the definitions, so we leave their proofs to the reader. Let us focus on statement 6. Consider the domains $\Om_j$ as in Definition~\ref{def} and the sets $\tilde\Om_j\subset \Om_j$ consisting of ergodic invariant tori. It suffices to show that meas$\,(\tilde\Om_j\cap \cF)=0$ for each $j$. 
 In the coordinates $\big(x,y,z\big)$,
corresponding  to $\Om_j$, denote by $\pi_z$ the natural $z$-projection. 
Assume the contrary, i.e. that  the measure above is nonzero, then 
meas$\,(J)>0$, where $J:=\pi_z(\tilde \Om_j\cap \cF)$.  Choose  any three points $z_1<z_2<z_3$ in $J$. Since $z_1, z_3\in \pi_z(\cF)$,
then for suitable $t_1,t_3$ we have $\pi_z(\xi(t_1))<z_2< \pi_z(\xi(t_3))$.  By continuity, there exists $t_2\in(t_1,t_3)$ such that 
$\xi(t_2)\in (\pi_z)^{-1}(z_2)=:T^2$.  Since $T^2$ is an ergodic invariant torus under the flow of $V$, then $\xi(t)\in T^2$ for all $t$. Hence 
 $J=\{z_2\}$, which is of measure zero. This contradiction proves statement 6.
\end{proof}

The following theorem establishes that  the functional $\vk(V)$ is continuous at $V$ if the vectorfield
 is integrable and nondegenerate. This property, which will be the key in our study of the 3D~Euler dynamics, is a consequence of the KAM theorem, stated in Section~\ref{s2}, and generally fails at points which are not integrable nondegenerate vectorfields. Moreover, if $V$ is analytic we prove that $\vk$ is H\"older continuous at $V$, 
  using the properties of  analytic functions to control the contribution to $\vk(W)$ from a
neighborhood of the singular set of $V$. We recall that, by definition, a function is analytic on a closed set if it is analytic in a neighborhood of the set.

\begin{theorem}\label{t2} Let $V\in\SVect^{k}_{ex}(M) $ be an integrable nondegenerate vectorfield. Then the functional $\vk$ is continuous at $V$ in the $C^{k}$-topology, provided that $k> 3$. Moreover, if $V$ is analytic ($C^\omega$) and well integrable, then $\vk$ is H\"older-continuous at $V$ with some exponent $\theta>0$:
\begin{equation} \label{Hol}
\vk(V)=1\ge
\vk(W)\geq 1-C_V\|V-W\|^{\theta}_{C^{k}}\,,
\end{equation}
for all $W\in     \SVect^{k}_{ex}(M)$.
\end{theorem}

\begin{proof}
Let $V$ be an integrable nondegenerate vector field and 
$W\in\SVect^{k}_{ex}(M)$ an exact divergence-free vectorfield that is close to $V$. By definition we have that $\vk(V)=1$. Consider a domain $\Om_j$ as in Definition~\ref{def} and set $\epsilon:=\|V-W\|_{C^{k}}$. We now take a subset $\tilde\Om_j(\delta)\subset\Om_j$ for any small $\delta>0$ such that:
\begin{itemize}
\item $\text{meas}\,(\Om_j\backslash \tilde\Om_j(\delta))\to 0$ as $\delta\to 0$.
\item $V$ is canonically integrable in $\tilde\Om_j(\delta)$.
\item $V$ satisfies the uniform twist condition $2$ of Section~\ref{s2}, cf. the inequality~\eqref{2}, with $\tau=\delta$, i.e.
\begin{equation}\label{equni}
|f'(z)g(z)-f(z)g'(z)|\geq \delta>0 \,.
\end{equation}
\end{itemize}
Then the KAM theorem~\ref{t1} implies that the contribution $\vkj(W)$ 
 to $\vk(W)$, coming from a set $\tilde\Om_j(\delta)$, is 
\begin{equation}\label{eqkam1}
\vkj(W)\ge \text{meas}\,(\tilde\Om_j(\delta))- C_j\delta^{-1}\|V-W\|^{1/2}_{C^{k}}\,,
\end{equation}
where $C_j$ is a $\delta$-independent 
 positive constant.  In order to get an estimate for $\vk(W)$ 
that just depends on $\epsilon$, we relate the small parameters $\delta$ and  $\epsilon$ by 
choosing $\delta=\epsilon^{1/2-\theta_1}$ for some  $\theta_1\in(0,1/2)$. 
Accordingly, denoting by $\Omega_j^\epsilon$ the set $\tilde\Om_j(\epsilon^{1/2-\theta_1})$ for each $j$, we get from Eq.~\eqref{eqkam1} that 
\begin{equation}\label{eqint}
\vkj(W)\ge\big( \text{meas}\,(\Om_j)-\text{meas}\,(\Om_j\backslash\Omega_j^\epsilon)-C_j\epsilon^{\theta_1}\big)
\longrightarrow\text{meas}\, (\Om_j)\quad\text{as} \quad \epsilon\to0\,. 
\end{equation}
Since  $\vk(W) \ge  \sum_j\vkj(W)$ and 
$\sum_j \text{meas}\,(\Om_j)=1$, then summing  relations \eqref{eqint} in $j$ 
we conclude that
$$
1\ge\vk(W)\longrightarrow 1\qquad \text{as}\quad \epsilon\longrightarrow 0. 
$$
That is,  $\vk$ is continuous at $V$. 

If $V$ is an analytic vectorfield, then  the measure of the set $\Om_j\backslash\Omega_j^\epsilon$ where Eq.~\eqref{equni} fails, satisfies 
\begin{equation}\label{eqkam2}
\text{meas}\,(\Om_j\backslash\Omega_j^\epsilon)\leq C_j\epsilon^{\theta_2}\,
\end{equation}
for some $\theta_2>0$, on account of Lojasiewicz's vanishing theorem~\cite[Section~6.3]{KP}, which we apply to the analytic function 
 in the LHS of Eq.~\eqref{equni}. Therefore, Eq.~\eqref{eqint} and~\eqref{eqkam2} imply that
\begin{equation}\label{eqkam3}
\vkj(W)\ge \text{meas}\,(\Om_j)-C_j\epsilon^{\theta_3}\,.
\end{equation}
Finally, summing up for finitely many $j$ in Eq.~\eqref{eqkam3} (we recall that now  $V$ is well integrable), we conclude that
$$
1\ge  \vk(W)\ge  1-C_V\epsilon^\theta     \,,
$$
thus proving Eq.~\eqref{Hol}. 
\end{proof}

\begin{corollary}\label{csigma}
If $M$ is analytic and   $V\in\SVect^{k}_{ex}(M)$ is analytic on the closure $\overline \Om_j$ 
of some domain $\Om_j$,  where $V$ is
canonically integrable and nondegenerate, then 
$$
\vk(W)\ge \text{meas}\,(\Om_j) -C_V \|V-W\|^{\theta}_{C^{k}} \qquad
\forall\,W\in     \SVect^{k}_{ex}\,,
$$
for some $\theta>0$ and a positive constant $C_V$. If the manifold $M$ is smooth and the vectorfield $V$ is not assumed to be analytic in $\overline \Om_j$, then
 $$\vk(W)\ge \text{meas}\,(\Om_j) 
  - E( \|V-W\|_{C^{k}} )\,,$$
where $E(t)$ is a continuous function satisfying $\lim_{t\to 0}E(t)=0$.
\end{corollary}

We finish this section by observing that on any closed $3$-manifold $M$ there is a nontrivial vector field $W\in \SVect^{k}_{ex}(M)$ such that $\vk(W)=0$. Indeed, take a couple of functions $(f_1,f_2):M\to\R^2$ which are independent almost everywhere in $M$, i.e. $\text{rank}(df_1(x),df_2(x))=2$ for all $x\in M$ except for a zero-measure set. Of course such a couple exists on any analytic $M$, e.g. two generic $C^\omega$ functions. Now we define $W$ as the unique vectorfield such that:
$$
i_W\mu=df_1\wedge df_2\,.
$$
It is obvious that $W$ is divergence-free and exact because $i_W\mu=d(f_1df_2)$, and that $f_1,f_2$ are first integrals of $W$. It then follows from Proposition~\ref{p}, item $4$, that $\vk(W)=0$.  

In contrast, the following seemingly obvious fact is still an open problem: to prove that there are open domains in $ \SVect^{k}_{ex}(M)$ where $\vk<1/2$.
To bypass this difficulty, in the next section we shall define and study a version of the partial integrability functional $\vk$, where the 
isotopy type of the invariant tori is fixed.

\bigskip


\section{Isotopy classes of invariant tori}   \label{isotopy}
In the previous section we have introduced the partial integrability functional $\vk$ which gives the measure of ergodic invariant tori, without distinguishing between different isotopy classes. Now we are going to exploit the different ways an invariant torus can be embedded in $M$ and define a countable number of such quantities.
We recall that two embedded tori $T^2_0$ and $T^2_1$ are \emph{isotopic} if there exists a family of embedded tori $T^2_t$, $t\in[0,1]$, connecting $T^2_0$ and $T^2_1$. It is well known that this property is equivalent to the existence of an isotopy $\Theta_t:M\times[0,1]\to M$ such that $\Theta_0=id$ and $\Theta_1(T^2_0)=T^2_1$~\cite{Hae61}. This equivalence relation defines the set of isotopy classes of embedded tori in $M$, we denote this set by $\cI(M)$. 
It is well known that the set of isotopy classes $\cI(M)$ is countable. 

The invariant tori in each domain $\Om_j$ introduced in Definition~\ref{def} have the same \emph{isotopy class}, but 
the class 
 can vary for different values of $j$. In this section, to be more precise, for each isotopy class $a\in\cI(M)$ we consider the sets $\Om_j^a$ which are the domains $\Om_j$ whose invariant tori are in the isotopy class $a$. Accordingly, now one can define a \emph{family of functionals} $\vk_a$, $a\in\cI(M)$, on the space of  exact divergence-free vectorfields.

\begin{definition}\label{defka} 
Given an isotopy class $a\in \cI(M)$, the \emph{partial integrability functional}
$$
\vk_a:\SVect^1_{ex}(M)\to [0,1]
$$
assigns to an exact   $C^1$-smooth divergence-free vectorfield $V$  the inner
measure of the set of ergodic 
$V$-invariant two-dimensional $C^1$-tori lying in the isotopy class $a$. The sequence 
 $$
  \cI(M) \ni a \mapsto \vk_a(V)
 $$
 is called the {\it integrability spectrum of $V$}. 
\end{definition}

 By definition one has  
 \begin{equation}\label{bydef}
  \vk(V)\geq \sum_a \vk_a(V)\,
   \end{equation}
 for the quantity $\vk(V)$ introduced before, since the inner measure is superadditive.

\begin{remark}\label{Rsame}
It is easy to check that items 1 and 3--6 of Proposition~\ref{p}, and Corollary~\ref{csigma} hold true if we substitute $\vk$ by $\vk_a$ and $\Om_j$ by $\Om_j^a$. Moreover, Theorem~\ref{t2} also holds for $\vk_a$, where the estimate~\eqref{Hol} now takes the form
$$
|\vk_a(W)-\vk_a(V)|\leq C_V\|V-W\|^{\theta}_{C^{k}}\qquad
\forall\,W\in     \SVect^{k}_{ex}(M)\,.
$$
\end{remark}

Remark~\ref{Rsame} and \eqref{bydef} 
imply  that the integrability spectrum is continuous at points $V$ which are
integrable nondegenerate vectorfields (here $k>3$):

\begin{proposition}\label{prop1}
Let $V\in \SVect^{k}_{ex}(M)$ be an integrable nondegenerate vectorfield. Then for each $a\in \cI(M)$ the function $V\mapsto \vk_a(V)$
is continuous at $V$. 
\end{proposition}

In the following lemma we introduce a subset $\cI_0(M)$ consisting of tori, lying in a fixed $3$-ball,
embedded in $M$. 
This subset of $\cI(M)$ is key to prove Theorem~\ref{Tconst} below. Roughly speaking the lemma shows that knotted tori in a ball cannot be unknotted in $M$.

\begin{lemma}\label{L:Tiso}
Let $M$ be a closed $3$-manifold and $B\subset M$ a $3$-ball. Two-tori embedded in $B$ whose core knots are neither isotopic nor mirror images in $B$, are not isotopic in $M$. In particular, the subset $\cI_0(M)\subset\cI(M)$ of isotopy classes of such tori is isomorphic to $\Z$. 
\end{lemma}   
\begin{proof}
Let us take two knots $L_1,L_2$ contained in $B\subset M$. Recall that a knot is a smoothly embedded circle. For each $i\in\{1,2\}$, define a torus $T_i^2\subset B\subset M$, which is the boundary of a tubular neighborhood $N(L_i)$ of the knot $L_i\subset B$, i.e. $T_i^2=\partial N(L_i)$. Let us prove that if $T_1^2$ and $T_2^2$ are isotopic in $M$, then the knots $L_1$ and $L_2$ are isotopic in $B$ or they are mirror images of each other. First, $T_1^2$ and $T_2^2$ being isotopic it follows that there is a diffeomorphism $\Theta':M\backslash T_1^2\to M\backslash T_2^2$, and we claim that this implies that $M\backslash \overline{N(L_1)}$ is diffeomorphic to $M\backslash \overline{N(L_2)}$. Indeed, the manifold $M\backslash T_i^2$ consists of two connected components, that is $N(L_i)$ and $M\backslash \overline {N(L_i)}$, so the existence of $\Theta'$ implies that:
\begin{itemize}
\item either there is a diffeomorphism $\Theta:M\backslash \overline{N(L_1)}\to M\backslash \overline{N(L_2)}$, as desired,
\item or diffeomorphisms $\Theta_1:M\backslash \overline{N(L_1)}\to N(L_{2})$ and 
$\Theta_2:M\backslash \overline{N(L_2)}\to N(L_{1})$.
\end{itemize}
In the second case, since there is a diffeomorphism $\Phi':N(L_2)\to N(L_1)$ because both sets are solid tori, we conclude that the diffeomorphism $\Theta_2^{-1}\circ \Theta'\circ \Theta_1$ transforms $M\backslash \overline{N(L_1)}$ onto $M\backslash \overline{N(L_2)}$, as we wanted to prove.

Accordingly, $M\backslash \overline{N(L_1)}$ is diffeomorphic to $M\backslash \overline{N(L_2)}$, so performing the connected sum prime decomposition~\cite{Ha07} of $M\backslash \overline{N(L_i)}$, its uniqueness readily implies that there exists a diffeomorphism $\Phi:B\backslash \overline{N(L_1)} \to B\backslash \overline{N(L_2)}$. It is easy to see that $B\backslash \overline{N(L_i)}$ is diffeomorphic to $B\backslash L_i$, so we get that $B\backslash L_1$ is diffeomorphic to $B\backslash L_2$. We conclude from Gordon-Luecke's theorem~\cite{GL89} that either $L_1$ and $L_2$ are isotopic in $B$ provided that  $\Phi$ is orientation-preserving, or they are mirror images of each other otherwise.

The previous discussion implies that different isotopy classes of knots in $B$ that are not mirror images, define different isotopy classes of embedded tori in $M\supset B$, where the tori are just the boundaries of tubular neighborhoods of the knots. Therefore, one can define each element in $\cI_0(M)$ as the set of embedded tori in a $3$-ball whose core knots are either isotopic or mirror images. It is standard that the set $\cI_0(M)$ is isomorphic to $\Z$, see e.g.~\cite{Ro03}.    
\end{proof}

\begin{remark}
In general, the set of isotopy classes $\cI(M)$ is bigger than $\cI_0(M)$, e.g. an embedded torus $T^2$ can be homologically nontrivial in $M$, that is $0\neq [T^2]\in H_2(M;\Z)$, so different homology classes give rise to different isotopy classes. Of course $\cI(M)=\cI_0(M)$ if e.g. $M=\mathbb S^3$. 
\end{remark}

While the index $a\in\cI(M)$ takes values in the set of all isotopy classes of embedded tori in $M$, for our purposes often 
 it suffices to assume that $a$ takes values in the subset $\cI_0(M)$ of embedded tori in a $3$-ball of $M$, i.e. $a\in\cI_0(M)\cong\Z$, cf. Lemma~\ref{L:Tiso}.
The existence of ``many'' invariant tori of a vectorfield, taking into account their isotopy classes, will be exploited in our analysis of the Euler equation below. 

\begin{example}\label{ex3}
\rm{
In Example~\ref{extorus} of an integrable vectorfield on $\T^3$, all invariant tori
of the field $\rot u^z$  are defined by $\{z=c\}$, and hence they are in the same isotopy class, call it $a_0\in\cI(\T^3)$. Therefore, $\vk_{a_0}(\rot u^z)=1$ and $\vk_{a}(\rot u^z)=0$ for all $a\neq a_0$. As we discussed in Example~\ref{extorus},  the isotopy class $a_0$ is nontrivial, and in fact it does not belong to $\cI_0(\T^3)$, cf. Lemma~\ref{L:Tiso}.
}
\end{example}

We finish this section by constructing an exact divergence-free vectorfield $V_a$, for any $a\in\Z$, on a closed $3$-manifold $M$ which is canonically integrable and nondegenerate on a domain $\Om_1^a\subset M$ whose measure can be made arbitrarily  close to $1$. Moreover, the $V_a$-invariant tori in $\Om_1^a$ correspond to the isotopy class $a\in\cI_0(M)$. 

\begin{theorem}\label{Tconst}
Let $M$ be closed $3$-manifold with volume form $\mu$. Then, for any $a\in\cI_0(M)\cong\Z$ and $0<\delta<1$ there exists a partially integrable vectorfield $V_a\in \SVect^{\infty}_{ex}(M)$ such that:
\begin{itemize}
\item It admits a $V_a$-invariant domain $\Om_1^a$ with meas$\,(\Om_1^a)=1-\delta$.
\item $V_a|_{\Om_1^a}$ is canonically integrable and nondegenerate, and so $\vk_a(V_a)\geq 1-\delta$.
\item $V_a|_{\overline\Om_1^a}$ can be taken analytic if $M$ and $\mu$ are analytic. 
\end{itemize}
\end{theorem}
\begin{proof}
We divide the construction of the vectorfield $V_a$ in three steps:
\\

\noindent {\bf Step 1}: Let $L_a$ be a knot in a $3$-ball $B$. Take a solid torus $\Omega_a\subset B$ which is a tubular neighborhood of the knot $L_a$, and hence diffeomorphic to $\es^1\times (0,1)^2$. We assume that for different $a\in\Z$, the knots $L_a$ are neither isotopic nor mirror images in $B$. It is easy to see that for any $0<\delta<1$, one can smoothly glue a big ball $\tilde B$ to $\Omega_a$ to get a new domain, which we still denote by $\Omega_a$, which is a solid torus isotopic to the original one, $L_a$ being at its core, and
\begin{equation}\label{eqball}
\text{meas}\,(B\backslash \Omega_a)= \frac{\delta}{4}\,.  
\end{equation}
Next, we embed the $3$-ball $B$ in $M$ in such a way that Eq.~\eqref{eqball} holds (with respect to the volume form $\mu$) and
\begin{equation}\label{eqM}
\text{meas}\,(M\backslash B)= \frac{\delta}{4}\,.
\end{equation}
If the manifold $M$ is analytic, the submanifolds $L_a$, $\Omega_a$ and $B$ can be slightly perturbed to make them analytic~\cite{Hirsch}, keeping Eqs.~\eqref{eqball} and~\eqref{eqM}. Let $N(L_a)$ be a small closed tubular neighborhood of the curve $L_a$ of measure
\begin{equation}\label{eqmetub}
\text{meas}\,(N(L_a))=\frac{\delta}{4}\,,
\end{equation}
and define the set $\hat\Om_1^a:=\Omega_a\backslash N(L_a)\subset B\subset M$. This set is obviously fibred by tori belonging to the same isotopy class $a\in\cI_0(M)$ in $M$ on account of Lemma~\ref{L:Tiso} (so for different $a$, the isotopy classes of the tori are different). Proceeding as in Section~\ref{s2} we parameterize the domain $\hat\Om_1^a$ with coordinates $(x,y,z)\in \T^2\times (-1,1)$. One can choose these coordinates in such a way that the volume form $\mu|_{\hat\Om_1^a}=dx\wedge dy\wedge dz$. Indeed, for general coordinates the volume form has the expression $\mu|_{\hat\Om_1^a}=p(x,y,z)dx\wedge dy\wedge dz$ for some positive function $p$, so defining a new variable $\int_{-1}^z p(x,y,s)ds$, which we still call $z$, one obtains the desired coordinate system.
\\

\noindent {\bf Step 2}: Now we construct a smooth vectorfield $V_a$ on $M$ that is divergence-free with respect to the volume form $\mu$ and that is canonically integrable and nondegenerate in a domain $\Om^1_a\subset\hat\Om^1_a$. This vectorfield can be easily defined using the local coordinates $(x,y,z)$ by the expression:
$$
V_a := \left\{
  \begin{array}{l l}
    f(z)\partial_x+g(z)\partial_y & \quad \text{in}\,\, \hat\Om_1^a\,,\\
    0 & \quad \text{in}\,\, M\backslash\hat\Om_1^a\,, 
  \end{array}\right.
$$
where the functions $f,g$ are smooth, satisfy the twist condition~\eqref{2} with a constant $\tau(c)>0$ in each interval $[-c,c]$ for $c<1$, and are chosen in such a way that they glue smoothly with $0$ as $z\to\pm 1$. By construction, the tori $T^2(c):=\T^2\times\{c\}$ are $V_a$-invariant and nondegenerate for $c\in (-1,1)$. It is obvious from the expression of the volume form $\mu$ in the coordinate $(x,y,z)$ that 
\begin{equation}\label{eqdiv2}
di_{V_a}\mu=0\,. 
\end{equation}
Accordingly, defining a $V_a$-invariant set $\Om_1^a\subset\hat\Om_1^a$, expressed in the coordinates $(x,y,z)$ as $\T^2\times(-c_0,c_0)$ for some $c_0<1$, and such that $\text{meas}\,(\hat\Om_1^a\backslash\Om_1^a)=\delta/4$, then the set $\Om_1^a$ has measure
$$
\text{meas}\,(\Om_1^a)=1-\delta
$$
by Eqs.~\eqref{eqball}--\eqref{eqmetub}. Moreover, $V_a$ is canonically integrable and nondegenerate in $\Om_1^a$, so $\vk_a(V_a)\geq 1-\delta$. If the manifold and the volume form are analytic, it is clear that the vectorfield $V_a$ can be taken analytic in $\overline\Om_1^a$.
\\

\noindent {\bf Step 3}: It remains to prove that $V_a$ is exact, that is the $2$-form $\beta:=i_{V_a}\mu$ is exact. Hodge decomposition explained in Section~\ref{s1} implies that this is equivalent to 
$$
\int_M h\wedge \beta=0
$$
for any closed $1$-form $h$ on $M$. To prove this, we notice that $\beta$ is supported in the solid torus $\Omega_a\subset B$, and $h=dR$ in the $3$-ball $B$ for some function $R\in C^\infty(B)$, because any closed form is exact in a contractible domain. So we have
\begin{align*}
\int_M h\wedge \beta &=\int_{\Omega_a}h\wedge \beta=\int_{\Omega_a} dR\wedge\beta=\int_{\Omega_a}d(R\beta)-\int_{\Omega_a} Rd\beta\\ &=\int_{\partial \Omega_a}R\beta=0\,,
\end{align*}
where we have used Eq.~\eqref{eqdiv2}, Stokes theorem and the fact that $\beta=0$ in $\partial \Omega_a$. This completes the proof of the theorem.  
\end{proof}

\bigskip


\section{A non-ergodicity theorem for the 3D Euler equation}   \label{s4}
Our goal in this section is to apply the previously developed machinery to study the evolution of the Euler equation~\eqref{4.1} on a closed $3$-manifold $M$.
For a non-integer $k>1$ the classical result of Lichtenstein 
 (see e.g. in~\cite{EM70}) shows that the Euler equation defines a local flow $\{\cS_t\}$
of  homeomorphisms of the space $\SVect^k(M)$, where for any $u\in\SVect^k(M)$ 
the solution 
$$u(t,\cdot)=\cS_t(u)\,,\qquad \cS_0(u)=u\,,$$ 
is defined for $t_*(u)<t<t^*(u)$ and is $C^1$-smooth in $t$. It is unknown whether the numbers $t_*<0$ and $t^*>0$ are finite.  

For this $u$ and for 
 $t_*<t<t^*$ we denote by $S_0^t:M\to M$ 
the (non-autonomous) flow of the equation
$$
\dot x=u(t,x),\qquad x\in M\,,
$$
which describes the Lagrangian map of the fluid flow.
According to Kelvin's circulation theorem (see e.g.~\cite{AKh}),
the corresponding vorticity field $\rot u(t)$ is  transported by the fluid flow:
\begin{equation}\label{Kelv}
\rot  u(t) =S_{0 *}^t\big( \rot u(0)\big)\,.
\end{equation}

Since the maps $S_0^t$ are volume-preserving  $C^{k}$-diffeomorphisms of $M$, the item~3 in Proposition~\ref{p} and Remark~\ref{Rsame} imply the following.

\begin{theorem}\label{t3}
 If $u \in\SVect^k(M)$ for $k>2$ and non-integer, then $\vk_a (\rot \cS_t(u))=\,$const for all $a\in \cI(M)$. In other words, the integrability spectrum of $\rot u$, that is 
$$a\in\cI(M) \mapsto \vk_a(\rot(u))$$ 
is an integral of motion of the Euler equation on the space $\SVect^k(M)$. 
\end{theorem}
The classical conserved quantities of the Euler equation (e.g.~\cite[Section I.9]{AKh}) are the \emph{energy},
 $$\cE(u):=\int_M (u,u)\,\mu\,,$$
 and the \emph{helicity}
 $$\cH(u):=\int_M (u,\rot u)\,\mu\,.$$  
 (Note that in terms of the vorticity field $\omega=\rot u $ the helicity assumes the  form
$\cH(u)=\hat\cH(\omega):=\int_M (\rot^{-1}\omega,\omega)\,\mu$).
Theorem~\ref{t3} introduces other conserved quantities $\vk_a, a\in \cI(M)$,  of the Euler flow. 
Below we are going to make use of the continuity property of these functionals. 

\smallskip

The following lemma is a version of Theorem~\ref{Tconst} where we fix the energy and  helicity, and construct a vectorfield with  prescribed values of those quantities, as well as a prescribed value of a partial integrability functional $\vk_a$.

\begin{lemma}\label{Leh}
Let $M$ be a closed $3$-manifold endowed with a volume form $\mu$, and fix arbitrary $h\in \R$ 
and a sufficiently large $e>0$. Then, for any $a\in \cI_0(M)\cong\Z$ and $0<\delta<1$ there exists a vectorfield $u_a\in \SVect^{\infty}(M)$ such that:
\begin{enumerate}
\item $\rot u_a$ has an invariant domain $\Om_1^a$ with meas$\,(\Om_1^a)=1-\delta$.
\item $\rot u_a|_{\Om_1^a}$ is canonically integrable and nondegenerate, so that 
$\vk_a(\rot u_a)\geq 1-\delta$.
\item $\rot u_a|_{\overline\Om_1^a}$ can be taken analytic provided that $M$ and $\mu$ are analytic.
\item $\cH(u_a)=h$ and $\cE(u_a)=e$.
\end{enumerate}
\end{lemma}
\begin{proof}
For each $a\in\cI_0(M)$, let $V_a$ be the exact smooth 
divergence-free vectorfield constructed in Theorem~\ref{Tconst}, which is supported on the set $\tilde\Om_1^a$ and is analytic in a set $\overline\Om_1^a$ provided that the manifold $M$ and the volume form $\mu$ are analytic. Applying Lemma~\ref{l11} we obtain that there exists a vectorfield $u'_a\in \SVect^{\infty}(M)$ such that $\rot u'_a=V_a$. The properties of the vectorfield $V_a$ imply that the above conditions $1$, $2$ and $3$  hold, so it remains to prove that one can modify $u'_a$ to fulfill condition $4$ without altering the other conditions.

Indeed, define the vectorfield $u_a:=u'_a+v+\lambda v'$, $\lambda\in\R$, where $v$ and $v'$ are divergence-free vectorfields supported on sets $K$ and $K'$ respectively such that $\overline K\cap \overline{\Omega_a}=\overline K'\cap \overline{\Omega_a}=\overline K\cap \overline K'=\emptyset$. We also assume that the helicity $\cH(v')=0$. Then
$$\cH(u_a)=\cH(u'_a)+\cH(v)\,,$$
where we have used that 
$$
\int_M v\rot u'_a=\int_Mv'\rot u'_a=\int_Mv\rot v'=\int_Mv'\rot v=0\,,
$$ 
which, in turn, holds since the supports of $v,v'$ and $\rot u'_a$ are pairwise disjoint, and that 
$$\int_M u'_a\rot v=\int_M v\rot u'_a=0\,,
$$ 
integrating by parts, and the same for $v'$. Since the helicity of $v$ can take any real value, it follows that one can choose it so that $\cH(u_a)=h$. Regarding the energy we have
$$\cE(u_a)=\cE(u'_a+v)+2\lambda\int_Mv'(u'_a+v)+\lambda^2\cE(v')\,,$$
and hence choosing appropriate $\lambda$ and $v'$ we get that $\cE(u_a)=e$ for an arbitrary real
$e\ge \cE(u'_a+v)$. By construction, $\rot u_a=V_a$ in the domain $\Omega_a$, which completes the proof of the lemma.  
\end{proof}

\begin{remark}
For a given vector field $\omega=\rot u$ its helicity and energy satisfy the Schwartz inequality
$$
|\hat\cH(\omega)|=\Big|\int_M (u,\omega)\,\mu \Big|\le C\int_M(\omega,\omega)\,\mu=C\cE(\omega)\,,
$$ 
where the constant $C$ depends on the Riemannian geometry of the manifold $M$, as described by Arnold~\cite{Arn73}, see also~\cite[Chapter III]{AKh}.
This constant is the maximal absolute value of the eigenvalues of the (bounded) operator $\rot^{-1}$ on exact divergence-free vectorfields. In terms of the velocity field $u$, when comparing 
the helicity $\cH(u)$ and the energy $\cE(u)$ we are dealing with the unbounded operator $\rot$, and so the above inequality is no longer relevant. 
\end{remark}

Now we are in a position to prove the non-ergodicity of the Euler flow. Everywhere below we assume that $ k>4$ is a non-integer number.

\begin{theorem}\label{thm:main}
Let $M$ be a closed $3$-manifold with a volume form $\mu$. 
Fix two  constants $h\in \R$ and  $e\gg 1$.  
Then there is a non-empty open set $\Gamma_a\subset \SVect^{k}(M)$ for each $a\in\Z$ such that $\cH(\Gamma_a)=h$, $\cE(\Gamma_a)=e$, and
$$
\Gamma_a\cap\cS_t(\Gamma_{b})=\emptyset
$$ 
for $a\neq b$ and for all $t$ for which the local flow is defined.
\end{theorem}

\begin{proof}
Take the vectorfields $u_a\in\SVect^{k}(M)$ constructed in Lemma~\ref{Leh}, all of them of fixed energy $e$ and helicity $h$. Since $\text{meas}\,(\Om^1_a)=1-\delta$ and $\rot u_a|_{\Om^1_a}$ is nondegenerate, Corollary~\ref{csigma} and Remark~\ref{Rsame} imply that any vectorfield $u'_a\in \Gamma_a$ which is $\epsilon$-close to $u_a$ in the $C^{k}$ topology satisfies
$$\vk_a(\rot u'_a)\geq 1-\delta-E(\epsilon)\,,$$
where $E$ is an error function satisfying $\lim_{\epsilon\to 0}E(\epsilon)=0$, and hence $\vk_b(\rot u'_a)\leq \delta+E(\epsilon)$ for any $b\neq a$, in view of \eqref{bydef}. 
Since the quantity $\vk_a(\rot u'_a)$ is conserved by the Euler flow, cf. Theorem~\ref{t3}, the theorem follows by taking $\delta$ and $\epsilon$ sufficiently small.  
\end{proof}

This theorem implies that the dynamical system defined by the Euler flow~\eqref{4.1} in the space $\Lambda_{e,h}\subset\SVect^{k}(M)$ of fixed energy $e\gg 1$ and helicity $h$, is neither ergodic nor mixing. The reason is that there are open sets of $\Lambda_{e,h}$ which do not intersect under the evolution of the Euler equation. We recall that according to a result of N.~Nadirashvili~\cite{Na91} (see also~\cite[Section II.4.B]{AKh}) the dynamical system defined by the 2D Euler equation on an annulus has  wandering trajectories in the $C^1$-topology, which is the strongest form of non-ergodicity. 
Namely, {\it for the Euler equation in a 2D annulus there is a divergence-free vectorfield $u_0$, such that for all fields in a sufficiently small $C^1$-neighborhood $\mathcal U$ of 
$u_0$ and sufficiently large time $T$ the Euler flow $\mathcal S_t$
after this time is ``never returning": $\mathcal U\cap S_t(\mathcal U)=\emptyset$ for all $t>T$}, \cite{Na91}. A similar result was proved by A.~Shnirelman~\cite{Sh97} showing the existence of wandering solutions for the Euler equation on the $2$-torus in appropriate Besov spaces.
Our result can be regarded as a 3D version of the 2D Nadirashvili's and Shnirelman's theorems.

Our next theorem provides a criterion to guarantee that neighborhoods (in the $C^{k}$-topology) of two divergence-free vectorfields do not intersect under the Euler flow. The result is stated in terms of the integrability spectrum introduced in Definition~\ref{defka}.

\begin{theorem}\label{Thlast}
Let $u,v \in \SVect^{k}(M) $ be vectorfields such that $\rot u$ and $\rot v$ are integrable and nondegenerate. Assume that 
the integrability spectra of $\rot u$ and $\rot v$ are different. Then the orbit of a sufficiently small  $C^{k}$-neighbourhood of $u$ under the Euler flow stays at a positive $C^{k}$-distance
from $v$.
\end{theorem}
 \begin{proof}
 By assumption, there is some $a  \in\cI(M)$ for which $\vk_a(\rot v)>\vk_a(\rot u)$. Then Proposition~\ref{prop1} implies that this inequality remains true for all vectorfields $v'\in \mathcal U(v)$ and $u'\in \mathcal U(u)$, where $\mathcal U(v)$ and $\mathcal U(u)$ are sufficiently small  $C^{k}$-neighbourhoods of $v$ and $u$. Since $\vk_a(\rot u')$ is a conserved quantity of the Euler flow according to Theorem~\ref{t3}, the claim follows.
\end{proof} 
 
If $u\in\SVect^{k}(M)$ is an analytic divergence-free vectorfield, one can get an explicit estimate for the $C^{k}$-distance  between $u$ and a trajectory $\cS_t(v)$ of the Euler flow, as stated in the following theorem.
 
\begin{theorem}\label{p2}
Let 
$u,v \in \SVect^{k}(M) $ be vectorfields such that $\rot u$ and $\rot v$ are canonically integrable and 
nondegenerate in domains $\Om^a_1$ and $\Om^b_1$, $a\neq b$, respectively. 
We also assume that $u$ is analytic in $\overline\Om^a_1$ and that $\text{meas}\,(\Om_1^a)=:\sigma_a(u)>1/2$ and $\text{meas}\,(\Om^b_1)=:\sigma_b(v)>1/2$.
Then there are constants $\eta>0$ and $C_u$ (the latter depending only on $u$) such that for all $t$ one has the estimate
\begin{equation}\label{esti}
\dist_{C^{k}}(u, \cS_t(v))\ge C_u(\sigma_a(u)+\sigma_b(v)-1)^{\eta}>0\,. 
\end{equation}
\end{theorem}

\begin{proof}
The assumptions imply that $\vk_a(\rot u)\geq \sigma_a(u)>1/2$ and $\vk_b(\rot v)\geq \sigma_b(v)>1/2$,
so $\vk_a(\rot v)\leq 1-\vk_b(\rot v)<1/2$. 
The theorem then follows from Corollary~\ref{csigma} and Remark~\ref{Rsame}, where $\eta=1/\theta$.
\end{proof}

In Examples~\ref{extorus} and~\ref{exsphere} (see also Example~\ref{ex3}) we have constructed analytic steady solutions of the Euler equation in $\T^3$ and $\mathbb S^3$ whose vorticities are integrable and nondegenerate vectorfields. Applying Theorem~\ref{p2} one can
estimate from below  the $C^{k}$-distance  of these steady solutions $u$ to the trajectories of the Euler flow for many initial conditions. For instance, take a vectorfield $v$ as in Lemma~\ref{Leh}, analytic in $\overline\Om_1^a$ and whose invariant tori are non-trivially knotted, and let $w$ be a field from 
 an $\epsilon$-neighborhood $\mathcal U(v)$ of $v$ in the $C^{k}$-topology. 
In this case, $\vk_{a_0}(\rot u)=1$ and $\vk_a(\rot w)\geq 1-\delta-C_v\epsilon^{\theta}$ by Corollary~\ref{csigma} and Remark~\ref{Rsame} (here $a_0\neq a$ denotes the isotopy class of the invariant tori of the steady state), so we conclude that
$$
\dist_{C^{k}}(u, \cS_t(w))\ge C_u(1-\delta-C_v\epsilon^{\theta})^{\eta}\,.
$$
In fact, in Example~\ref{extorus} we have constructed steady solutions $u^x$, $u^y$ and $u^z$ on $\T^3$ whose invariant tori (as well as invariant tori of their vorticities) are given by $\{x=c\}$, $\{y=c\}$ and $\{z=c\}$, respectively, so they are not isotopic (and homologically distinct). Therefore, the steady states $u^x$, $u^y$ and $u^z$ have $C^{k}$-neighbourhoods ${\mathcal U}_x$, ${\mathcal U}_y$ and ${\mathcal U}_z$
such that a trajectory of the 3D Euler equation cannot pass through  two different neighbourhoods, i.e. these neighbourhoods do not mix under the Euler flow. In particular, no two of these steady states  can be joined by a heteroclinic connection.

\begin{remark}\label{remverylast}
One should mention that there are other dynamical properties of $C^1$ vectorfields that are invariant under diffeomorphisms and that can be used to analyze the behavior of the 
3D Euler equation, similarly to the invariants $\vk_a$  introduced in this paper. 
A natural extension of the functionals $\vk_a$ consists in considering similar invariants 
for sets of embedded tori, forming non-trivial links instead of knots discussed above. 
In this case, for a given number $N$ of link components, the index $a$ runs over 
isotopy classes  of configurations of $N$ embedded non-intersecting tori in $M$. 
This gives more elaborate examples of mutually avoiding open sets.

Another extension is given by considering the rotation numbers of the invariant tori. 
Such invariants allow one to ``localize'' the conserved quantities $\vk_a$ by taking those 
invariant tori whose rotation number belongs to a certain interval.

In a different spirit, one can also consider the functional
$n(V)$, defined as the number of singular points of the field $V$
if all the singularities are hyperbolic and as infinity otherwise, and it 
is obviously invariant under diffeomorphisms. The hyperbolic permanence theorem implies that $n(V)$ is locally constant at $V$ if $n(V)<\infty$, so proceeding as we explain in this section, one could prove the non-ergodicity of 3D Euler using $n$ instead of $\vk_a$. The main advantages of the invariants $\vk_a$ compared with $n$ (and other invariants of vectorfields) are:
\begin{itemize}
\item It is easy to show that the conserved quantities $\vk_a$ are independent of the energy and the helicity (see Lemma~\ref{Leh}). Additionally, the construction of exact divergence-free vectorfields with prescribed $\vk_a$ is not very hard (see Theorem~\ref{Tconst}), while it is not clear how to construct exact divergence-free vectorfields  whose all singularities are hyperbolic.
\item The invariants $\vk_a$ are lower semicontinuous and are particularly  well behaved for vectorfields that are close to nondegenerate integrable stationary solutions of the Euler equation, which are ``typical'' according to Arnold's structure theorem. The latter allows one to analyze the role of these steady states for the long time dynamics of 3D Euler in $C^{k}$, $k>4$.
\end{itemize}
\end{remark}

%
%


\begin{thebibliography}{1000} 

\bibitem{Arn66}
V.I. Arnold, 
\textit{Sur la g\'eometrie diff\'erentielle des groupes de Lie de dimension
infinie et ses application \`a l'hydrodynamique des fluides parfaits}, Ann. Inst.
Fourier (Grenoble) 16 (1966) 319--361.

\bibitem{Arn73}
V.I. Arnold, 
\textit{The asymptotic Hopf invariant and its applications},  Proc. Summer School in Diff. Eq.
(1974); Engl. transl.:  Selecta Math. Soviet. 5 (1986) 327--345. 

\bibitem{AKh}
V.I. Arnold and B.A. Khesin,
\textit{Topological Methods in Hydrodynamics}, Springer-Verlag, 
New York 1998.

\bibitem{Bo}
J. Bourgain and D. Li, 
\textit{Strong ill posedness of the incompressible Euler equation in integer  $C^m$ spaces},
 arXiv:1405.2847.

\bibitem{Sevr}
H.W. Broer, G.B. Huitema and M.B. Sevryuk,
\textit{Quasi-periodic motions in families of dynamical systems. Order amidst chaos},
Lecture Notes in Mathematics, vol.~1645, Springer-Verlag, 
Berlin 1996.

\bibitem{BHT}
H.W. Broer, G.B. Huitema and F. Takens, \textit{Unfoldings of quasi-periodic tori}, Mem. Amer. Math. Soc. 83 (1990) 1--81.

\bibitem{C-S}
C.-Q. Cheng and Y.-S. Sun, 
\textit{Existence of invariant tori in three-dimensional measure-preserving maps}, 
Cel. Mech. Dyn. Ast. 47 (1990) 275--292.

\bibitem{C-W}
C.-Q. Cheng and L. Wang
\textit{Destruction of Lagrangian torus for positive definite hamiltonian systems}, 
GAFA 23 (2013) 848-866.



\bibitem{Do82}
R. Douady, \textit{Une d\'emonstration directe de l'\'equivalence des th\'eor\`emes de tores invariants pour diff\'eomorphismes et champs de vecteurs}, C. R. Acad. Sci. Paris 295 (1982) 201--204.

\bibitem{EM70}
D.G. Ebin and J. Marsden, 
\textit{Groups of diffeomorphisms and the motion of an incompressible fluid},
Ann. of Math. 92 (1970) 102--163.

\bibitem{GL89}
C. Gordon and J. Luecke, \textit{Knots are determined by their complements}, J. Amer. Math. Soc. 2 (1989) 371--415.

\bibitem{Hae61}
A. Haefliger, \textit{Plongements diff\'erentiables de vari\'et\'es dans vari\'et\'es}, Comment. Math. Helv. 36 (1961) 47--82.

\bibitem{Ha07}
A. Hatcher, \textit{Notes on basic 3-manifold topology}, available electronically at http://www.math.cornell.edu/~hatcher/.

\bibitem{He83}
M. Herman, \textit{Sur les courbes invariantes par les diff\'eomorphismes de l'anneau, Vol 1}, Ast\'erisque 103-104 (1983) 1--221.

\bibitem{Hirsch}
M. Hirsch, \textit{Differential Topology}, Springer-Verlag, New York 1976.

\bibitem{KP}
S. Krantz and H. Parks, \textit{A primer of real analytic functions}, Birkh\"auser, Boston 2002. 

\bibitem{Mo65}
J. Moser, \textit{On the volume elements on a manifold}, Trans. Amer. Math. Soc. 120 (1965) 286--294.

\bibitem{Na91}
N. Nadirashvili, \textit{Wandering solutions of the two-dimensional Euler equation}, Funct. Anal. Appl. 25 (1991) 220--221.

\bibitem{Pos}
J. P\"oschel, 
\textit{Integrability of Hamiltonian systems on Cantor sets}, Comm. Pure Appl. Math.  35 (1982) 653--695.

\bibitem{Ro03}
D. Rolfsen, \textit{Knots and links}, AMS, Providence 2003.

\bibitem{Sal}
D.  Salamon,
\textit{The Kolmogorov-Arnold-Moser theorem},  Mathematical Physics Electronic Journal 10 (2004) 1--37. 

\bibitem{Sh97}
A. Shnirelman, \textit{Evolution of singularities, generalized Liapunov function and generalized integral for an ideal incompressible fluid}, Amer. J. Math. 119 (1997) 579--608.

\bibitem{Taylor} 
M. Taylor, 
\textit{Partial Differential Equations}, 
Springer-Verlag, Berlin 1996. 

\bibitem{W}
R. O. Wells, 
  \textit{Differential Analysis on Complex Manifolds},
  Springer-Verlag,  Berlin 1980.
%


\end{thebibliography}
\end{document}